\documentclass[11pt]{amsart}
\usepackage{amssymb}
\usepackage[margin=1in]{geometry}
\usepackage[colorlinks]{hyperref}
\newtheorem{theorem}{Theorem}[section]
\newtheorem{lem}[theorem]{Lemma}
\newtheorem{prop}[theorem]{Proposition}
\newtheorem{cor}[theorem]{Corollary}

\theoremstyle{definition}
\newtheorem{definition}[theorem]{Definition}
\newtheorem{example}[theorem]{Example}

\theoremstyle{remark}
\newtheorem{remark}[theorem]{Remark}

\numberwithin{equation}{section}
\begin{document}

\newcommand{\spacing}[1]{\renewcommand{\baselinestretch}{#1}\large\normalsize}
\spacing{1.14}

\title[Riemannian geometry of tangent Lie groups]{Riemannian geometry of tangent Lie groups using two left invariant Riemannian metrics}

\author {M. Hassanvand}

\address{Morteza Hassanvand \\ Department of Pure Mathematics \\ Faculty of  Mathematics and Statistics\\ University of Isfahan\\ Isfahan\\ 81746-73441-Iran.}
\email{morteza.hassanvand1991@gmail.com and m{\_}hsnvnd@sci.ui.ac.ir}

\author {H. R. Salimi Moghaddam}
\address{Hamid Reza Salimi Moghaddam \\ Department of Pure Mathematics \\ Faculty of  Mathematics and Statistics\\ University of Isfahan\\ Isfahan\\ 81746-73441-Iran.\\ Scopus Author ID: 26534920800 \\ ORCID Id:0000-0001-6112-4259}
\email{salimi.moghaddam@gmail.com and hr.salimi@sci.ui.ac.ir}


\date{\today}

\begin{abstract}
In this paper, we consider a Lie group $G$ equipped with two left-invariant Riemannian metrics $g^1$ and $g^2$. Using these two left-invariant Riemannian metrics we define a left-invariant Riemannian metric $\tilde{g}$ on the tangent Lie group $TG$. The Levi-Civita connection, tensor curvature, and sectional curvature of $(TG,\tilde{g})$ in terms of $g^1$ and $g^2$ are given. Also, we give a sufficient condition for $\tilde{g}$ to be bi-invariant. Finally, motivated by the recent work of D. N. Pham, using symplectic forms $\omega _1$ and  $\omega _2$ on $G$ we define a symplectic form $\tilde{\omega}$ on $TG$.

\end{abstract}

\maketitle

\section*{\textbf{Introduction}}
The horizontal and vertical lifts of geometric structures on the tangent bundle of a Riemannian manifold play an important role in the study of dynamical systems of free-fall particle motion, relativistic charged particles, and Newtonian systems (see \cite{Maartens-Taylor} and \cite{Sarbach-Zannias}). In \cite{Oliver-Davis}, the authors used the lift of vectors in examining thermodynamic or kinetic theory models for matter field space-times. \\
On the other hand, in \cite{Crampin}, Crampin studied time-independent Lagrangian dynamics using vertical and complete lifts. In this work we study the geometric properties of lifted left invariant Riemannian metrics on the tangent bundle of a Lie group arising from two left invariant Riemannian metrics, using vertical and complete lifts. In addition to the physical applications, the results of this article also generalize the results of \cite{Asgari-Salimi1}.\\
The study of the Riemannian geometry of tangent bundle $TM$ of a Riemannian manifold $M$ goes back to the fundamental work written by Sasaki \cite{Sasaki}. He constructed a Riemannian metric $\hat{g}$ on the tangent bundle $TM$ by using the Riemannian metric $g$ of the base manifold $M$ (for more details see \cite{Albuquerque}). This construction is based on the natural splitting of the tangent bundle as follows
$$TTM=HTM \oplus VTM,$$
where $HTM$ and $VTM$ are horizontal and vertical subbundles of $TTM$ respectively. In 1962, Dombrowski obtained an explicit formula for the Levi-Civita connection $\widehat{\nabla}$ and the curvature tensor $\widehat{R}$ of the  Riemannian manifold $(TM,\hat{g})$ (see \cite{Dombrowski}). Yano and Kobayashi, in \cite{Yano-Kobayashi1,Yano-Kobayashi2,Yano-Kobayashi3}, have described in detail the lift of tensor fields and connections. In the midst of it, they have mentioned some geometric properties of the tangent bundle of Lie groups, for example, they showed that if $X$ is a left invariant vector field on a Lie group, then vertical and complete lifts of $X$ (i.e., $X^v$ and $X^c$ respectively) are also left invariant vector fields on the tangent bundle $TG$, and by way of this, they obtained the connection between the Lie algebra of $G$ and the Lie algebra of $TG$. Cheeger and Gromoll provided a new way to construct a metric on the tangent bundle in 1972 \cite{Cheeger-Gromoll}, then Musso and Tricerri gave an explicit formula for some natural metrics \cite{Musso-Tricerri}. In 1986, Sekizawa generalized the complete lift for reductive homogenous and symmetric spaces in \cite{Sekizawa}. Sekizawa and Kowalski in \cite{Kowalski-Sekizawa} have provided an explicit  description of the natural transformation of Riemannian metrics on a smooth manifold $M$ to  Riemannian metrics on its tangent bundle $TM$.\\
In recent years, many research works have been done in this field, one can mention the joint works of Gudmundson and Kappos \cite{Gudmundson-Kappos1,Gudmundson-Kappos2}, who investigated the geometry of the tangent bundle where $TM$ is equipped with Sasaki and Cheeger-Gromoll metrics. In \cite{Asgari-Salimi1}, for a Lie group $G$ and a left invariant Riemannian metric $g$, the lifted metric $\tilde{g}$ on $TG$, using vertical and complete lifts of left invariant vector fields, is defined  as follows:
\begin{align*}
&\tilde{g}(X^c,Y^c)=g(X,Y),\\
&\tilde{g}(X^v,Y^v)=g(X,Y),\\
&\tilde{g}(X^c,Y^v)=0,
\end{align*}
where $X$ and $Y$ are arbitrary left invariant vector fields on $G$.\\
In this work, we define a metric $\tilde{g}$ on $TG$ using two left invariant metrics $g^1$ and $g^2$ on $G$. Then we will explore the geometric properties of $(TG,\tilde{g})$based on the geometric properties of $(G,g^1)$ and $(G,g^2)$. In \cite{Pham}, a symplectic form $\tilde{\omega}$ on $TG$ is defined by a symplectic form $\omega$ on $G$ as follows:
\begin{align*}
&\tilde{\omega}(X^c,Y^c)=\pi^*\big(\omega(X,Y)\big),\\
&\tilde{\omega}(X^c,Y^v)=\pi^*\big(\omega(X,Y)\big),\\
&\tilde{\omega}(X^v,Y^v)=0,
\end{align*}
where $X$ and $Y$ are arbitrary left invariant vector fields on $G$ and $\pi:TG\longrightarrow G$ is the projection map. Finally, motivated by \cite{Pham}, we will define a symplectic form $\tilde{\omega}$ on $TG$ based on two symplectic forms $\omega_1$ and $\omega_2$ on $G$.


\section{\textbf{Preliminaries}}
Here, we will provide some basic definitions and results concerning Lie groups and the geometry of tangent bundles.\\
Assume that $M$ is an $n$-dimensional smooth manifold and $C^{\infty}(M)$ and $\mathfrak{X}(M)$ denote the spaces of smooth real-valued functions and smooth vector fields on $M$, respectively.
Suppose that $\varphi: \Bbb{R}\times M\longrightarrow M$ is a one-parameter group action on $M$, then $\varphi$ corresponds to a smooth vector field as follows:
$$X_{\varphi}f (x)= \frac{d}{dt}\lvert _{t=0} (f \circ \varphi_{t})(x) \ \ \ ; \ \ \ \forall f \in C^{\infty}(M).$$
We should note for a fixed point  $p \in M$, $t \mapsto \varphi_{t}(p)$ is an integral curve of $X_{\varphi}$, therefore, one can consider $X \in \mathfrak{X}(M)$ as an equivalence class of curves $\varphi_{t}(x)$, i.e., $X(x)=[\varphi_{t}(x)]$. Suppose that $X(x)=[\varphi_{t}(x)]$ is a smooth vector field on $M$. The vertical lift of $X$ is the vector field $X^v \in \mathfrak{X}(TM)$ generated by the equivalence class of curves
$$\Phi _{t}(x,y)=y+t X(x) \ \ \ \forall y\in T_xM.$$
Also, the complete lift of $X$ is the vector field $X^c \in \mathfrak{X}(TM)$ generated by the equivalence class of curves $T\varphi_{t}(x)$.\\
For any two vector fields $X$ and $Y$ on the manifold $M$, the Lie bracket of the vertical and complete lifts of them satisfies the following properties \cite{Yano-Ishihara}:
\begin{align}
&[X^c,Y^c]=[X,Y]^c,\nonumber\\
&[X^c,Y^v]=[X,Y]^v, \label{Lie brackets}\\
&[X^v,Y^v]=0.\nonumber
\end{align}
Throughout this paper, $G$ denotes an $n$-dimensional Lie group with the multiplication map $\mu: G \times G \to G$ given by $\mu (g,h)=gh$ and the inversion map $\iota: G \to G$ given by $\iota (g)=g^{-1}$ and the identity element $e$. For any $g \in G$, the left and right translation maps denote by $L_g$ and $R_g$. Now, using the identification $T(G\times G) \cong TG \times TG$, we have:
$$T\mu : TG \times TG \to TG$$
$$(v,w)\mapsto T\mu (v,w)= T_h L_g w + T_g R_h v,$$
for all $v\in T_g G$ and $w\in T_h G$. We can see $(TG,T\mu)$ is a Lie group with identity element $0_e \in T_e G$ and inversion map $T\iota$ (see \cite{Neeb}). We denote the Lie algebras of the Lie groups $G$ and $TG$ by $\mathfrak{g}$ and $\tilde{\mathfrak{g}}$, respectively. It has been shown that the complete and vertical lifts of any left invariant vector field $X\in\mathfrak{g}$ belong to the Lie algebra $\tilde{\mathfrak{g}}$ (see \cite{Hin,Yano-Ishihara}).
Indeed, if $\{X_1,\dots , X_n\}$ is a basis of $\mathfrak{g}$, then $\{X_1^v,\dots ,X_n^v,X_1^c,\dots , X_n^c \}$ is a basis of $\tilde{\mathfrak{g}}$, because the linear map
$$\xi : \mathfrak{g} \oplus \mathfrak{g} \to \tilde{\mathfrak{g}}$$
given by
\begin{equation}\label{1.4}
\xi(X,Y)=X^c+Y^v
\end{equation}
is an isomorphism.

\section{\textbf{Lifted metric and its Levi-Civita connection}}

In this section, using two arbitrary left invariant Riemannian metrics $g^1$ and $g^2$ on the Lie group $G$, we define a left invariant Riemannian metric $\tilde{g}$ on the Lie group $TG$. We also calculate its Levi-Civita connection $\tilde{\nabla}$.
This definition generalizes the definition given in \cite{Asgari-Salimi1}.
\begin{definition}
Let $g^1$ and $g^2$ be any two left invariant Riemannian metrics on the Lie group $G$. Define a left invariant Riemannian metric $\tilde{g}$ on $TG$ as follows:
\begin{equation*}
\forall X,Y \in {\mathfrak{g}}\ \ \ \ ; \ \ \ \
\begin{cases}
&\tilde{g}(X^c,Y^c)=g^1(X,Y),\\
&\tilde{g}(X^v,Y^v)=g^2(X,Y),\\
&\tilde{g}(X^c,Y^v)=0.
\end{cases}
\end{equation*}
\end{definition}
Suppose that $\{X_1,\dots,X_n\}$ is a basis of $\frak{g}$. From the matrix point of view, by considering the basis $\{X_1^v,\dots,X_n^v,X_1^c,\dots,X_n^c \}$ for $\tilde{\mathfrak{g}}$ we can write
$$\tilde{g}=\begin{pmatrix}
g^2&0
\\0&g^1
\end{pmatrix}.$$
In the following, we will study the Riemannian geometry of $(TG,\tilde{g})$ and its relationship
with the Riemannian geometry of $(G,g^1)$ and $(G,g^2)$. We present a lemma similar to Lemma 2.1
of \cite{Asgari-Salimi1}.
\begin{lem}
Let $\nabla^1$, $\nabla^2$ and $\tilde{\nabla}$ be the Levi-Civita connections of the Riemannian manifolds $(G,g^1)$, $(G,g^2)$ and $(TG,\tilde{g})$, respectively. Then for any $X$, $Y$ and $Z$ in $\mathfrak{g}$ we have\\
\\
(1) $\tilde{g}(\tilde{\nabla}_{X^c}Y^c,Z^c)=g^1({\nabla^1}_{X}Y,Z)$,\\
(2) $\tilde{g}(\tilde{\nabla}_{X^c}Y^c,Z^v)=\tilde{g}(\tilde{\nabla}_{X^c}Y^v,Z^c)=\tilde{g}(\tilde{\nabla}_{X^v}Y^c,Z^c)=\tilde{g}(\tilde{\nabla}_{X^v}Y^v,Z^v)=0$,\\
(3) $\tilde{g}(\tilde{\nabla}_{X^c}Y^v,Z^v)=\frac{1}{2}\Big(g^2(Y,[Z,X])+g^2(Z,[X,Y])\Big)$,\\
(4) $\tilde{g}(\tilde{\nabla}_{X^v}Y^c,Z^v)=\frac{1}{2}\Big(g^2(Z,[X,Y])-g^2(X,[Y,Z])\Big)$,\\
(5) $\tilde{g}(\tilde{\nabla}_{X^v}Y^v,Z^c)=\frac{1}{2}\Big(g^2(Y,[Z,X])-g^2(X,[Y,Z])\Big)$.\\
\end{lem}
\begin{proof}
The proof is similar to Lemma 2.1 of \cite{Asgari-Salimi1}, so we omit it.
\end{proof}
The following remark plays an important role in our study.
\begin{remark}
On a given smooth manifold, each Riemannian metric defines an inner product on the tangent space at any point of it. If we have two Riemannian metrics $g^1$ and $g^2$ on the Lie group $G$, then
$$
{< ,>}_1:=g^1: \mathfrak{g}\times \mathfrak{g} \to \Bbb{R} \ \ and \ \ {< ,>}_2:=g^2: \mathfrak{g}\times \mathfrak{g} \to \Bbb{R},
$$
are two inner products on $\mathfrak{g}$, therefore there exists a symmetric linear map $\varphi: (\mathfrak{g},{< ,>}_2)\to(\mathfrak{g},{< ,>}_1) $ such that
$${<X ,Y>}_2 = {<\varphi (X) ,Y>}_1 \ \ \ \ \forall X,Y \in \mathfrak{g}.$$
Since $\varphi$ is symmetric, there is an orthonormal basis (with respect to ${< ,>}_1$)
\begin{equation}\label{2.1}
\mathcal{B}_1=\{X_1,\dots,X_n\}
\end{equation}
of eigenvectors of $\varphi$ for $\mathfrak{g}$. Hence, we have an orthonormal basis for $\tilde{\mathfrak{g}}$ as follows:
\begin{equation}\label{2.2}
\tilde{\mathcal{B}}=\{\frac{X_1^v}{\sqrt{\lambda_1}},\dots , \frac{X_n^v}{\sqrt{\lambda_n}},X_1^c,\dots , X_n^c\},
\end{equation}
where  $\lambda_1 , \dots , \lambda_n$ are eigenvalues associated with eigenvectors $X_1,\dots,X_n$.
\end{remark}

Now, we set the structure constants with respect to the basis $\mathcal{B}_1$ as follows:
$$\forall X_i,X_j \in \mathcal{B}_1  \ \ \ \ ;\ \ \ \ [X_i,X_j] =\sum_{k=1}^n c_{ij}^k X_k.$$
Also we set
$$\nabla_{X_i}^2 X_j =\sum_{k=1}^n \Gamma _{ij}^{k}X_k,$$
where $\Gamma _{ij}^{k}\in \Bbb{R}$. Based on these notations, we prove the following theorem.
\begin{theorem}\label{thm2.4}
Let $X$ and $Y$ be arbitrary left invariant vector fields on $G$, then \\
(1) $\tilde{\nabla}_{X^c}Y^c = \big(\nabla_{X}^1Y\big)^c$,\\
(2) $\tilde{\nabla}_{X^c}Y^v=\big(\nabla_{X}^2Y+\frac{1}{2}(ad_2 Y)^*X\big)^v$,\\
(3) $\tilde{\nabla}_{X^v}Y^c=\big(\nabla_{X}^2Y+\frac{1}{2}(ad_2 X)^*Y\big)^v$.\\
(4) $\tilde{\nabla}_{X^v}Y^v=\Big(\varphi \big(\nabla_{X}^2Y-\frac{1}{2}[X,Y]\big)\Big)^c,$
where $(ad_2 X)^*Y$ is the transpose of $(ad X)Y$ with respect to ${< ,>}_2$.
\end{theorem}
\begin{proof}
We prove the last one. The others are similar.
To prove $(4)$, we assume that $\tilde{Z}$ is any left invariant vector field on $TG$. According to (\ref{1.4}), we may write $\tilde{Z}=Z_{1}^{c}+Z_{2}^{v}$ for some $Z_1$ and $Z_2$ in $\mathfrak{g}$. Now by using the previous lemma and Koszul formula we have:
\begin{align*}
2\tilde{g}(\tilde{\nabla}_{X^v}Y^v,\tilde{Z})&=2 \tilde{g}(\tilde{\nabla}_{X^v}Y^v,Z_{1}^{c})+2 \tilde{g}(\tilde{\nabla}_{X^v}Y^v,Z_{2}^{v})\\
&=(g^2(Y,[Z_1,X]) - g^2(X,[Y,Z_1])\\
&=2g^2(\nabla_{X}^2Y,Z_1) - g^2(Z_1,[X,Y])\\
&=2g^1\big(\varphi(\nabla_{X}^2Y),Z_1\big) - g^1\big(\varphi([X,Y],Z_1)\big)\\
&=2\tilde{g}\big({(\varphi(\nabla_{X}^2Y))}^c,Z_1^c\big) - \tilde{g}\big({(\varphi([X,Y]))}^c,Z_1^c)\big).
\end{align*}
Therefore
$$\tilde{\nabla}_{X^v}Y^v=\Big(\varphi \big(\nabla_{X}^2Y-\frac{1}{2}[X,Y]\big)\Big)^c$$
\end{proof}
\begin{cor}
If we suppose that $\mathcal{B}_1$ is the given basis in \eqref{2.1}, then for any two left invariant vector fields $X=\sum_{i=1}^{n}\mu^i X_i$ and $Y=\sum_{i=1}^{n}\xi^j X_j$ on $G$ we have
$$\tilde{\nabla}_{X^v}Y^v=\sum_{i,j,k}\mu^i \xi^j \lambda_k (\Gamma _{ij}^{k} -\frac{1}{2}c_{ij}^k) X_k^c,$$
where $\lambda^1,\dots,\lambda^n$ are eigenvalues associated with eigenvectors $X_1,\dots,X_n$.
\end{cor}
We mention that a vector field $X$ on a Riemannian manifold $(M,g)$ is called a geodesic vector field
if $\nabla_X X=0$, where $\nabla$ denotes the Levi-Civita connection of $M$.
\begin{cor}
Let $X$ be a left invariant vector field on the Lie group $G$. Then,
\begin{enumerate}
\item if $X$ is a geodesic vector field with respect to $\nabla^1$, then so is $X^c$,
\item if $X$ is a geodesic vector field with respect to $\nabla^2$, then so is $X^v$.
\end{enumerate}
\end{cor}
\begin{cor}\label{cor 2.7}
Suppose that $g^1$ and $g^2$ are any two bi-invariant metrics on the Lie group $G$, then we have
\begin{align*}
&\tilde{\nabla}_{X^v}Y^v=0,\\
&\tilde{\nabla}_{X^c}Y^c=\frac{1}{2}[X,Y]^c,\\
&\tilde{\nabla}_{X^v}Y^c=0,\\
&\tilde{\nabla}_{X^c}Y^v=[X,Y]^v.
\end{align*}
\end{cor}
Let $(M,g)$ be a Riemannian manifold. A vector field $X\in \mathfrak{X}(M)$ is said to be a
conformal vector field (with resp. to $g$) if $\mathcal{L}_X g = 2 \rho g$ for some $\rho \in C^\infty(M)$.
Moreover, $X$ is called a Killing vector field when $\rho \equiv  0$.
\begin{theorem}
Let $X$ be a left invariant vector field on $G$. We have\\
(i) $X$ is a conformal vector field with respect to both of the left invariant Riemannian metrics $g^1$ and $g^2$ if and only if $X^c$ is a conformal vector field on $(TG,T\tilde{\mu})$.\\
(ii) The vector field $X^v$ is a Killing vector field if and only if $X \in Z(\mathfrak{g})$ where
$$Z(\mathfrak{g})=\{ Y \in \mathfrak{g}\ | \ \forall X \in \mathfrak{g} \ ; \ [X,Y]=0\}.$$
\end{theorem}
\begin{proof}
Consider the bases $\mathcal{B}_1$ and $\tilde{\mathcal{B}}$ for $\mathfrak{g}$ and $\tilde{\mathfrak{g}}$, respectively.\\
To prove (i), assume that $X$ is a conformal vector field with respect to both of the left invariant Riemannian metrics $g^1$ and $g^2$, then we show that $X^c$ is a conformal vector field. We have
\begin{align*}
\mathcal{L}_{X^c}\tilde{g}(X_i^c,X_j^c)&=-\tilde{g}([X,X_i]^c,X_j^c)-\tilde{g}(X_i^c,[X,X_j^c])\\
&=-g^1([X,X_i],X_j)-g^1(X_i,[X,X_j])\\
&=\mathcal{L}_{X}g^1(X_i,X_j),
\end{align*}
$$
\mathcal{L}_{X^c}\tilde{g}(X_i^c,X_j^v)=-\tilde{g}([X,X_i]^c,X_j^v)-\tilde{g}(X_i^c,[X,X_j]^v)=0,
$$
and
\begin{align*}
\mathcal{L}_{X^c}\tilde{g}(X_i^v,X_j^v)&==-\tilde{g}([X,X_i]^v,X_j^v)-\tilde{g}(X_i^v,[X,X_j^v]\\
&=-g^2([X,X_i],X_j)-g^2(X_i,[X,X_j])\\
&=\mathcal{L}_{X}g^2(X_i,X_j).
\end{align*}
Conversely, assume that $X^c$ is a conformal vector field on $(TG,T\tilde{\mu})$, so $\mathcal{L}_{X^c}\tilde{g}=2\tilde{\rho} \tilde{g}$ for some $\tilde{\rho}\in C^\infty(TG)$. Hence, we have
$$\mathcal{L}_{X}g^1(X_i,X_j)=\mathcal{L}_{X^c}\tilde{g}(X_i^c,X_j^c)=2\tilde{\rho} \tilde{g}(X_i^c,X_j^c)=2\tilde{\rho}g^1(X_i,X_j),$$
$$\mathcal{L}_{X}g^2(X_i,X_j)=\mathcal{L}_{X^c}\tilde{g}(X_i^v,X_j^v)=2\tilde{\rho} \tilde{g}(X_i^v,X_j^v)=2\tilde{\rho}g^2(X_i,X_j).$$
It should be noted that $\tilde{\rho}$ on each fiber of $TG$ is constant, then $\tilde{\rho}\in C^\infty(G)$. A similar argument proves that (ii) is true (see \cite{Seifipour-Peyghan}).
\end{proof}
Here we recall a famous proposition of \cite{Oneill}. Suppose the Lie group $G$ is connected and $g$ is a left invariant Riemannian metric on $G$. Then the following propositions are equivalent (see \cite{Oneill}):
\begin{enumerate}
\item $g$ is a bi-invariant Riemannian metric.
\item $g(X,[Y,Z])=g([X,Y],Z)$ for all $X, Y, Z \in \mathfrak{g}$.
\item $\nabla_X Y = \frac{1}{2}[X,Y]$ for all $X, Y \in \mathfrak{g}$, where $\nabla$ is the Levi-Civita connection.
\end{enumerate}
In addition, we mention that one can define a symmetric connection (so-called canonical connection) on an arbitrary Lie group $G$ as follows:
$$\forall X,Y \in \mathfrak{g} \ \ \ ; \ \ \ \nabla_X Y=\frac{1}{2}[X,Y].$$
A Lie group $G$ is called $indecomposable$, if its Lie algebra $\mathfrak{g}$ is not a direct sum of lower dimensional algebras.
\\
The canonical connection $\nabla$ on an indecomposable Lie group $G$ is the Levi-Civita connection of some metric, say $g$, if and only if $g$ is a non-degenerate solution to
\begin{equation}\label{2.3}
g(\big[Z,[X,Y]\big],W)+g(Z,\big[W,[X,Y]\big])=0,
\end{equation}
where $X$, $Y$, $Z$ and $W$ are arbitrary left or right invariant vector fields (See Theorem 4.2 of \cite{Ghanam-Thompson-Miller}).
With the previous discussion, we obtain the following results.
\begin{theorem}
Let $G$ be an indecomposable connected Lie group and $g$ be a left invariant Riemannian metric on $G$. Then $g$ is a bi-invariant metric if and only if $g$ satisfies the equation \eqref{2.3}.
\end{theorem}
Now, assume that $g^1$ and $g^2$ are two bi-invariant Riemannian metrics on a connected Lie group $G$. As we can see from corollary \ref{cor 2.7}, $\tilde{g}$ is not necessarily bi-invariant, because
$$\tilde{\nabla}_{X^c}Y^v=[X,Y]^v\neq \frac{1}{2}[X,Y]^v.$$
\begin{theorem}
Let $G$ be a connected and indecomposable Lie group and $g^1$ and $g^2$ be two left invariant Riemannian metrics on $G$. Then $\tilde{g}$ is a  bi-invariant metric on $TG$ if $g^1$ is a bi-invariant metric and $g^2$ satisfies the following relation:
\begin{equation}\label{2.4}
g^2(\big[Z,[X,Y]\big],W)=0 \ \ ; \ \ \forall X, Y, Z, W \in \mathfrak{g}.
\end{equation}
\end{theorem}

\begin{proof}
We will show that $\tilde{g}$ satisfies equation (\ref{2.3}). First of all, we can see for all  $X, Y, Z$ and $W \in \mathfrak{g}$,
$$\tilde{g}(\big[Z^c,[X^c,Y^c]\big],W^c)+\tilde{g}(Z^c,\big[W^c,[X^c,Y^c]\big])=0,$$
since $g^1$ is a bi-invariant metric. Also, according to the hypothesis and the relations \eqref{Lie brackets}, we have
\begin{equation*}\label{2.5}
\tilde{g}(\big[Z^v,[X^c,Y^c]\big],W^v)+\tilde{g}(Z^v,\big[W^v,[X^c,Y^c]\big])=0,
\end{equation*}
\begin{equation*}\label{2.6}
\tilde{g}(\big[Z^c,[X^c,Y^v]\big],W^v)+\tilde{g}(Z^c,\big[W^v,[X^c,Y^v]\big])=0,
\end{equation*}
\begin{equation}\label{2.7}
\tilde{g}(\big[Z^v,[X^v,Y^c]\big],W^c)+\tilde{g}(Z^v,\big[W^c,[X^v,Y^c]\big])=0,
\end{equation}
\begin{equation*}\label{2.8}
\tilde{g}(\big[Z^v,[X^c,Y^v]\big],W^c)+\tilde{g}(Z^v,\big[W^c,[X^c,Y^v]\big])=0,
\end{equation*}
\begin{equation*}\label{2.9}
\tilde{g}(\big[Z^c,[X^v,Y^c]\big],W^v)+\tilde{g}(Z^c,\big[W^v,[X^c,Y^c]\big])=0,
\end{equation*}
The following is an example of how to check the equation (\ref{2.7}):
\begin{align*}
&\tilde{g}(\big[Z^v,[X^v,Y^c]\big],W^c)+\tilde{g}(Z^v,\big[W^c,[X^v,Y^c]\big])\\
&=\tilde{g}(\big[Z^v,[X,Y]^v\big],W^c)+\tilde{g}(Z^v,\big[W^c,[X,Y]^v\big])\\
&=\tilde{g}(0,W^c)+\tilde{g}(Z^v,\big[W,[X,Y]\big]^v)\\
&= 0 + g^2(Z,\big[W,[X,Y]\big])\\
&= 0+0 \\
&=0.
\end{align*}

In the rest of the proof, by the definition of $\tilde{g}$ and the properties of the Lie bracket, we have

\begin{align*}
&\tilde{g}(\big[Z^c,[X^c,Y^c]\big],W^v)+\tilde{g}(Z^c,\big[W^v,[X^c,Y^c]\big])=0,\\
&\tilde{g}(\big[Z^v,[X^c,Y^v]\big],W^v)+\tilde{g}(Z^v,\big[W^v,[X^c,Y^v]\big])=0,\\
&\tilde{g}(\big[Z^v,[X^v,Y^v]\big],W^v)+\tilde{g}(Z^v,\big[W^v,[X^v,Y^v]\big])=0,\\
&\tilde{g}(\big[Z^v,[X^v,Y^v]\big],W^c)+\tilde{g}(Z^v,\big[W^c,[X^v,Y^v]\big])=0,\\
&\tilde{g}(\big[Z^c,[X^v,Y^v]\big],W^c)+\tilde{g}(Z^c,\big[W^c,[X^v,Y^v]\big])=0,
\end{align*}
\begin{equation}\label{2.15}
\tilde{g}(\big[Z^c,[X^v,Y^c]\big],W^c)+\tilde{g}(Z^c,\big[W^c,[X^v,Y^c]\big])=0,
\end{equation}
\begin{align*}
&\tilde{g}(\big[Z^c,[X^v,Y^v]\big],W^v)+\tilde{g}(Z^c,\big[W^v,[X^v,Y^v]\big])=0,\\
&\tilde{g}(\big[Z^v,[X^c,Y^c]\big],W^c)+\tilde{g}(Z^v,\big[W^c,[X^c,Y^c]\big])=0,\\
&\tilde{g}(\big[Z^v,[X^v,Y^c]\big],W^v)+\tilde{g}(Z^v,\big[W^v,[X^v,Y^c]\big])=0,\\
&\tilde{g}(\big[Z^c,[X^c,Y^v]\big],W^c)+\tilde{g}(Z^c,\big[W^c,[X^c,Y^v]\big])=0.
\end{align*}
For example, in order to verify equation \eqref{2.15},
\begin{align*}
&\tilde{g}(\big[Z^c,[X^v,Y^c]\big],W^c)+\tilde{g}(Z^c,\big[W^c,[X^v,Y^c]\big])\\
&=\tilde{g}(\big[Z,[X,Y]\big]^v,W^c)+\tilde{g}(Z^c,\big[W,[X,Y]\big]^v)\\
&=0+0\\
&=0.
\end{align*}
\end{proof}
\begin{remark}
Note that if $g^1$ and $g^2$ are two left invariant Riemannian metrics on a connected Lie group $G$ and $\tilde{g}$ is a bi-invariant Riemannian metric on $TG$ as the lift of both of them, then $g^1$ and $g^2$ are bi-invariant Riemannian metrics on $G$.
\end{remark}
\section{\textbf{Curvature}}
In this section, we compute the curvature tensor and the sectional curvature of $(TG,\tilde{g})$ in terms of curvature tensors and sectional curvatures of $(G,g^1)$ and $(G,g^2)$. Let us denote the curvature tensor and the sectional curvature of $(G,g^i)$ by $R^i$ and $K^i$, respectively, for $i=1,2$. We also denote the curvature tensor and the  sectional curvature of $(TG,\tilde{g})$ by $\tilde{R}$ and $\tilde{K}$, respectively.
\begin{theorem}\label{thm3.1}
Let $X$, $Y$ and $Z$ be arbitary left invariant vector fields on the Lie group $G$. Then\\
\begin{align*}
(1)\ \tilde{R}(X^c,Y^c)Z^c&=\big(R^1(X,Y)Z\big)^c,\\
(2)\ \tilde{R}(X^c,Y^c)Z^v&=\big(R^2(X,Y)Z\big)^v+\Big(\frac{1}{2}\nabla_X^2(ad_2Z)^* Y+ \frac{1}{2}\big(ad_2(\nabla_Y^2Z + \frac{1}{2}(ad_2Z)^*Y\big)^*X\\
&-\frac{1}{2}\nabla_Y^2(ad_2Z)^*X-\frac{1}{2}\big(ad_2(\nabla_X^2Z+\frac{1}{2}+(ad_2Z)^*X\big)^*Y-\frac{1}{2}(ad_2Z)^*[X,Y]\Big)^v,\\
(3)\ \tilde{R}(X^v,Y^c)Z^c&=\big(R^2(X,Y)Z\big)^v+\Big(\frac{1}{2}\nabla_X^2\big((ad_2X)^*\nabla_Y^1Z\big)
-\frac{1}{2}\nabla_Y^2\big((ad_2X)^*Z\big)-\frac{1}{2}\big(ad_2 \nabla_X^2Z \\
&+\frac{1}{2}(ad_2X)^*Z\big)^*Y-\frac{1}{2}(ad_2[X,Y])^*Z\Big)^v,\\
(4)\ \tilde{R}(X^v,Y^v)Z^c&=\big(\varphi(R^2(X,Y)Z)\big)^c+\Big(\varphi \big(\nabla_{[X,Y]}^2Z+\frac{1}{2}\nabla_X^2(ad_2Y)^*Z\\
&-\frac{1}{2}[X,\nabla_Y^2Z+\frac{1}{2}(ad_2Y)^*Z]
-\frac{1}{2}(ad_2X)^*Z-\frac{1}{2}[Y,\nabla_X^2Z+\frac{1}{2}(ad_2X)^*Z]\big)\Big)^c,\\
\end{align*}
\begin{align*}
(5)\ \tilde{R}(X^v,Y^c)Z^v&=\big(\varphi(R^2(X,Y)Z)\big)^c+\Big(\frac{1}{2}\varphi \big(\nabla_X^2 (ad_2 Z)^*Y-[X,\nabla_Y^2 Z]-[X,(ad_2 Z)^*Y]\\
&+\big[[XY],Z\big]\big)-\big(\varphi(\nabla_Y^2\nabla_X^2 Z)\big)-\big(\nabla_Y^1\varphi(\nabla_X^2 Z)-\frac{1}{2}\nabla_Y^1\varphi([X,Z])\big)\Big)^c,\\
(6)\ \tilde{R}(X^v,Y^v)Z^v&=\Big(\nabla_X^2\big(\varphi(\nabla_Y^2 Z-\frac{1}{2}[Y,Z])\big) + \frac{1}{2}(ad_2 X)^*\varphi(\nabla_Y^2 Z-\frac{1}{2}[Y,Z])\Big)^v\\
&-\Big(\nabla_Y^2\big(\varphi(\nabla_X^2 Z-\frac{1}{2}[X,Z])\big) + \frac{1}{2}(ad_2 Y)^*\varphi(\nabla_X^2 Z-\frac{1}{2}[X,Z])\Big)^v.
\end{align*}
\end{theorem}
\begin{proof}
Based on Theorem \ref{thm2.4}, this is a direct calculation.
\end{proof}

\begin{theorem}
Let $X$ and $Y$ be two left invariant vector fields on the Lie group $G$. Then\\
\\
(i) If $\{X,Y\}$ is an orthonormal set with respect to $g^1$, then $\tilde{K}(X^c,Y^c)=K^1(X,Y)$.\\
(ii) If $\{X,Y\}$ is an orthonormal set with respect to $g^2$, then
\begin{align*}
\tilde{K}(X^v,Y^v)&=g^2\big(\nabla_X^2\varphi(\nabla_Y^2 Y),X\big) + \frac{1}{2}g^2\big((ad_2 X)^*\varphi(\nabla_Y^2 Y),X\big)-g^2\big(\nabla_Y^2\big(\varphi(\nabla_X^2 Y),X\big)\\
&+\frac{1}{2}g^2 \big([X,Y],X\big)- \frac{1}{2}g^2\big((ad_2 Y)^*\varphi(\nabla_X^2 Y-\frac{1}{2}[X,Z]),X\big).
\end{align*}
(iii) If $\{X,Y\}$ is an orthonormal set with respect to $g^2$, then
\begin{align*}
\tilde{K}(X^v,Y^c)&=\frac{1}{{\lVert Y\rVert }_1^2}K^2(X,Y)+\frac{1}{2{\lVert Y\rVert }_1^2} g^2\big(\nabla_X^2 ((ad_2X)^*\nabla_Y^1 Y),X\big)-\frac{1}{2{\lVert Y\rVert }_1^2} g^2\big(\nabla_Y^2((ad_2X)^*Y),X\big)\\
&-\frac{1}{2{\lVert Y\rVert }_1^2} g^2\big(ad_2 \nabla_X^2 Y +\frac{1}{2}(ad_2X)^*Y)^*Y,X \big)-\frac{1}{2{\lVert Y\rVert }_1^2} g^2\big(ad_2[X,Y])^*Y,X\big).
\end{align*}
\end{theorem}
\begin{proof}
We prove this theorem using only the Theorem (\ref{thm3.1}).\\
$(i)$\
\begin{align*}
\tilde{K}(X^c,Y^c) =\tilde{g}\Big(\tilde{R}(X^c,Y^c)Y^c,X^c\Big)&=\tilde{g}\big(R^1(X,Y)Y)^c,X^c\big)\\
&=g^1\big(R^1(X,Y)Y),X\big)\\
&=K^1(X,Y).
\end{align*}
$(ii)$
\begin{align*}
\tilde{K}(X^v,Y^v)&=\tilde{g}\Big(\tilde{R}(X^v,Y^v)Y^v,X^v\Big)\\
&=\tilde{g}\Big(\big(\nabla_X^2\varphi(\nabla_Y^2 Y)\big)^v + \frac{1}{2}\big((ad_2 X)^*\varphi(\nabla_Y^2 Y)\big)^v-\big(\nabla_Y^2\big(\varphi(\nabla_X^2 Y)\big)^v\\
&+\frac{1}{2}[X,Y]^v- \frac{1}{2}\big((ad_2 Y)^*\varphi(\nabla_X^2 Y-\frac{1}{2}[X,Z])\big)^v,X^v\Big)\\
&=g^2\big(\nabla_X^2\varphi(\nabla_Y^2 Y),X\big) + \frac{1}{2}g^2\big((ad_2 X)^*\varphi(\nabla_Y^2 Y),X\big)-g^2\big(\nabla_Y^2\big(\varphi(\nabla_X^2 Y),X\big)\\
&+\frac{1}{2}g^2 \big([X,Y],X\big)- \frac{1}{2}g^2\big((ad_2 Y)^*\varphi(\nabla_X^2 Y-\frac{1}{2}[X,Z]),X\big).\\
\end{align*}
$(iii)$
\begin{align*}
\tilde{K}(X^v,Y^c)&=\frac{1}{{\lVert Y\rVert }_1^2}\tilde{g}\Big(\tilde{R}(X^v,Y^c)Y^c,X^v\Big)\\
&=\frac{1}{{\lVert Y\rVert }_1^2}\tilde{g}\Big(\big(R^2(X,Y)Y\big)^v,X^v\Big)+\frac{1}{2{\lVert Y\rVert }_1^2}\tilde{g}\Big(\nabla_X^2\big((ad_2X)^*\nabla_Y^1Z\big),X^v\Big) \\
&-\frac{1}{2{\lVert Y\rVert }_1^2}\tilde{g}\Big(\nabla_Y^2\big((ad_2X)^*Z\big),X^v\Big)-\frac{1}{2{\lVert Y\rVert }_1^2}\tilde{g}\Big(\big(ad_2 \nabla_X^2Z +\frac{1}{2}(ad_2X)^*Z)^*Y\big)^v,X^v\Big) \\
&-\frac{1}{2{\lVert Y\rVert }_1^2}\tilde{g}\Big(\big(ad_2[X,Y])^*Z\big)^v,X^v\Big)\\
&=\frac{1}{{\lVert Y\rVert }_1^2}K^2(X,Y)+\frac{1}{2{\lVert Y\rVert }_1^2} g^2\big(\nabla_X^2 ((ad_2X)^*\nabla_Y^1 Y),X\big)-\frac{1}{2{\lVert Y\rVert }_1^2} g^2\big(\nabla_Y^2((ad_2X)^*Y),X\big)\\
&-\frac{1}{2{\lVert Y\rVert }_1^2} g^2\big(ad_2 \nabla_X^2 Y+\frac{1}{2}(ad_2X)^*Y)^*Y,X \big)-\frac{1}{2{\lVert Y\rVert }_1^2} g^2\big(ad_2[X,Y])^*Y,X\big).\\
\end{align*}

\end{proof}
\begin{cor}
Let $g^1$ be a bi-invariant Riemannian metric on a connected and indecomposable Lie group $G$. Suppose $g^2$ is a left invariant Riemannian metric on $G$ such that satisfies the equation \eqref{2.4}, then for all $\tilde{X}$, $\tilde{Y}$ and $\tilde{Z}$ in $\tilde{\mathfrak{g}}$
\begin{align*}
&{\tilde{\nabla}_{\tilde{X}}} \tilde{Y}=\frac{1}{2}[\tilde{X},\tilde{Y}],\\
&\tilde{R}(\tilde{X},\tilde{Y})\tilde{Z}=\frac{1}{4}\big[ [\tilde{X},\tilde{Y}],\tilde{Z}\big],\\
&\tilde{K}(\tilde{X},\tilde{Y})=\frac{1}{4} \ \frac{\tilde{g}\big([\tilde{X},\tilde{Y}],[\tilde{X},\tilde{Y}]\big)}{\tilde{g}(\tilde{X},\tilde{X}) \tilde{g}(\tilde{Y},\tilde{Y})- \big(\tilde{g}(\tilde{X},\tilde{Y})\big)^2}.\\
\end{align*}
\end{cor}
\section{\textbf{Some geometric properties in terms of structure constants}}
The goal of this section is to describe some geometric properties by using structure constants. In section 2, it has been noted that the basis $\mathcal{B}_1$ given in (\ref{2.1}) is orthonormal with respect to $g^1$. Therefore, $\mathcal{B}_2=\{\frac{X_1}{\sqrt{\lambda_1}},\dots , \frac{X_n}{\sqrt{\lambda_n}}\}$ is an orthonormal basis with respect to $g^2$. Assume that $d_{ij}^{k}$ denotes the structure constants with respect to this basis, so we have
\begin{equation}\label{4.1}
d_{ij}^{k} = \frac{\sqrt{\lambda_k}}{\sqrt{\lambda_i \lambda_j}} \ c_{ij}^{k}.
\end{equation}
Let us write $\{Y_1,\dots,Y_n,Y_{n+1},\dots,Y_{2n} \}$ for
$\tilde{\mathcal{B}}=\{\frac{X_1^v}{\sqrt{\lambda_1}},\dots , \frac{X_n^v}{\sqrt{\lambda_n}},X_1^c,\dots , X_n^c\}$, where
$$Y_i=\frac{X_i^v}{\sqrt{\lambda_i}}\ \ \ \ and \ \ \  Y_{i+n}=X_i^c  \ \ \ ; \ \ \ 1 \le i \le n.$$
We set
\begin{equation}\label{4.1}
[Y_i,Y_j] =\sum_{k=1}^{2n} b_{ij}^k Y_k.
\end{equation}
On the other hand the relations \eqref{Lie brackets} show that
\begin{align}
&[X_i^c,X_j^c]=\sum_{k=1}^n c_{ij}^k X_k^c, \label{4.2}\\
&[X_i^c,\frac{X_j^v}{\sqrt{\lambda_j}}]=\sum_{k=1}^n \frac{\sqrt{\lambda_k}}{\sqrt{\lambda_j}} \ c_{ij}^k \ \frac{X_k^v}{\sqrt{\lambda_k}}, \label{4.3}\\
&[\frac{X_i^v}{\sqrt{\lambda_i}},\frac{X_j^v}{\sqrt{\lambda_j}}]=0.\label{4.4}
\end{align}
Now, we consider the following cases:\\
\textbf{Case 1:} If $1 \le i,j \le n$, then for all $1 \le k \le 2n$ we have $b_{ij}^k=0$.\\
\textbf{Case 2:} If $1 \le i \le n$ \& $n+1 \le j \le 2n$, then
$$b_{ij}^k = \begin{cases} \frac{\sqrt{\lambda_k}}{\sqrt{\lambda_i}} \ c_{i \ j-n}^k & 1 \le k \le n \\ 0 & n+1 \le k \le 2n \end{cases}$$
\textbf{Case 3:} If $n+1 \le i \le 2n$ \& $1 \le j \le n$, then
$$b_{ij}^k = \begin{cases}
\frac{\sqrt{\lambda_k}}{\sqrt{\lambda_j}} \ c_{i-n \ j}^k & 1 \le k \le n \\ 0 & n+1 \le k \le 2n
\end{cases}$$
\textbf{Case 4:} If $n+1 \le i,j \le 2n$, then
$$b_{ij}^k = \begin{cases}
0 & 1 \le k \le n \\  c_{i-n \ j-n}^{k-n} & n+1 \le k \le 2n
\end{cases}$$
By using Koszul formula we can obtain the Levi-Civita connections $\nabla^1$ and $\nabla^2$ in terms of structure constants as follows:
\begin{equation}
\nabla_{X_i}^1 X_j = \frac{1}{2} \sum_{k=1}^{n} (c_{ij}^k - c_{jk}^i +c_{ki}^j) X_k,
\end{equation}
\begin{equation}
\nabla_{\frac{X_i}{\sqrt{\lambda_i}}}^2 \frac{X_j}{\sqrt{\lambda_j}} = \frac{1}{2} \sum_{k=1}^{n} \big(\frac{\sqrt{\lambda_k}}{\sqrt{\lambda_i \lambda_j}} \ c_{ij}^k - \frac{\sqrt{\lambda_i}}{\sqrt{\lambda_j \lambda_k}}\ c_{jk}^i + \frac{\sqrt{\lambda_j}}{\sqrt{\lambda_k \lambda_i}} \ c_{ki}^j \big) \frac{X_k}{\sqrt{\lambda_k}}.
\end{equation}
In the same way we have
\begin{equation}
\tilde{\nabla}_{Y_i} Y_j = \frac{1}{2} \sum_{k=1}^{n} (b_{ij}^k - b_{jk}^i + b_{ki}^j) Y_k.
\end{equation}
Considering the above discussion, we can deduce the following theorem.
\begin{theorem}\label{thm4.1}
As above, consider the orthonormal bases $\mathcal{B}_1$ and $\tilde{\mathcal{B}}$ for $\mathfrak{g}$ and $\tilde{\mathfrak{g}}$, respectively. Then
\begin{enumerate}
  \item $\tilde{\nabla}_{\frac{X_i^v}{\sqrt{\lambda_i}}} \frac{X_j^v}{\sqrt{\lambda_j}}=\frac{1}{2} \sum_{l=1}^{n}\big(\frac{\sqrt{\lambda_j}}{\sqrt{\lambda_i}} \ c_{li}^j - \frac{\sqrt{\lambda_i}}{\sqrt{\lambda_j}} \ c_{jl}^i\big) X_l^c,$
  \item $\tilde{\nabla}_{X_i^c}X_j^c = \frac{1}{2} \sum_{l=1}^{n} (c_{ij}^l - c_{jl}^i +c_{li}^j) X_l^c,$
  \item $\tilde{\nabla}_{X_i^c}\frac{X_j^v}{\sqrt{\lambda_j}} = \frac{1}{2} \sum_{l=1}^{n}\big(\frac{\sqrt{\lambda_l}}{\sqrt{\lambda_j}} \ c_{ij}^l + \frac{\sqrt{\lambda_j}}{\sqrt{\lambda_l}} \ c_{li}^j\big) \frac{X_l^v}{\sqrt{\lambda_l}},$
  \item $\tilde{\nabla}_{\frac{X_i^v}{\sqrt{\lambda_i}}}X_j^c = \frac{1}{2} \sum_{l=1}^{n}\big(\frac{\sqrt{\lambda_l}}{\sqrt{\lambda_i}} \ c_{ij}^l - \frac{\sqrt{\lambda_i}}{\sqrt{\lambda_l}} \ c_{jl}^i\big) \frac{X_l^v}{\sqrt{\lambda_l}}.$
\end{enumerate}
\end{theorem}
In two next theorems, we compute the curvature tensor and the sectional curvature in terms of structure constants.
\begin{theorem}\label{thm4.2}
By considering the given basis in (\ref{2.2}) for $\tilde{\mathfrak{g}}$, we have\\
\begin{align*}
(1)\ \tilde{R}(X_i^c,X_j^c)X_k^c&=\frac{1}{4} \sum_{h=1}^{n}\Big[ \sum_{l=1}^{n} (c_{jk}^l - c_{kl}^j +c_{lj}^k)  (c_{il}^h - c_{lh}^i +c_{hi}^l) \\
&-(c_{ik}^l - c_{kl}^i +c_{li}^k)  (c_{jl}^h - c_{lh}^j +c_{hj}^l)\\
&-2 c_{ij}^l (c_{lk}^h - c_{kh}^l +c_{hl}^k)\Big] X_h^c\\
(2)\ \tilde{R}(X_i^c,X_j^c)\frac{X_k^v}{\sqrt{\lambda_k}}&=\frac{1}{4}\sum_{h=1}^{n} \Big[ \sum_{l=1}^{n}\big(\frac{\sqrt{\lambda_l}}{\sqrt{\lambda_k}} \ c_{jk}^l + \frac{\sqrt{\lambda_k}}{\sqrt{\lambda_l}} \ c_{lj}^k\big) \big(\frac{\sqrt{\lambda_h}}{\sqrt{\lambda_l}} \ c_{il}^h + \frac{\sqrt{\lambda_l}}{\sqrt{\lambda_h}} \ c_{hi}^l\big) \\
&-\big(\frac{\sqrt{\lambda_l}}{\sqrt{\lambda_k}} \ c_{ik}^l + \frac{\sqrt{\lambda_k}}{\sqrt{\lambda_l}} \ c_{li}^k\big) \big(\frac{\sqrt{\lambda_h}}{\sqrt{\lambda_l}} \ c_{jl}^h + \frac{\sqrt{\lambda_l}}{\sqrt{\lambda_h}} \ c_{hj}^l\big) \\
&-2 c_{ij}^l \big(\frac{\sqrt{\lambda_h}}{\sqrt{\lambda_k}} \ c_{lk}^h + \frac{\sqrt{\lambda_k}}{\sqrt{\lambda_h}} \ c_{hl}^k\big)\Big] \frac{X_h^v}{\sqrt{\lambda_h}}\\
(3)\ \tilde{R}(\frac{X_i^v}{\sqrt{\lambda_i}},X_j^c)X_k^c &=\frac{1}{4} \sum_{h=1}^{n}\sum_{l=1}^{n} \Big[ (c_{jk}^l - c_{kl}^j +c_{lj}^k) \big(\frac{\sqrt{\lambda_h}}{\sqrt{\lambda_i}} \ c_{il}^h - \frac{\sqrt{\lambda_i}}{\sqrt{\lambda_h}} \ c_{lh}^i\big) \\
&-\big(\frac{\sqrt{\lambda_l}}{\sqrt{\lambda_i}} \ c_{ik}^l - \frac{\sqrt{\lambda_i}}{\sqrt{\lambda_l}} \ c_{kl}^i\big) \big(\frac{\sqrt{\lambda_h}}{\sqrt{\lambda_l}} \ c_{jl}^h + \frac{\sqrt{\lambda_l}}{\sqrt{\lambda_h}} \ c_{hj}^l\big)\\
&-2\frac{\sqrt{\lambda_l}}{\sqrt{\lambda_i}} \ c_{ij}^l \ \big(\frac{\sqrt{\lambda_h}}{\sqrt{\lambda_l}} \ c_{lk}^h - \frac{\sqrt{\lambda_l}}{\sqrt{\lambda_h}} \ c_{kh}^l\big) \Big] \frac{X_h^v}{\sqrt{\lambda_h}}.\\
(4)\ \tilde{R}(\frac{X_i^v}{\sqrt{\lambda_i}},\frac{X_j^v}{\sqrt{\lambda_j}})X_k^c&=\frac{1}{4} \sum_{h=1}^{n}\sum_{l=1}^{n}\Big[\big(\frac{\sqrt{\lambda_l}}{\sqrt{\lambda_k}} \ c_{jk}^l - \frac{\sqrt{\lambda_j}}{\sqrt{\lambda_l}} \ c_{kl}^j\big) \big(\frac{\sqrt{\lambda_l}}{\sqrt{\lambda_i}} \ c_{hi}^l - \frac{\sqrt{\lambda_i}}{\sqrt{\lambda_l}} \ c_{lh}^i\big) \\
&-\big(\frac{\sqrt{\lambda_l}}{\sqrt{\lambda_i}} \ c_{ik}^l - \frac{\sqrt{\lambda_i}}{\sqrt{\lambda_l}} \ c_{kl}^i\big) \big(\frac{\sqrt{\lambda_l}}{\sqrt{\lambda_j}} \ c_{hj}^l - \frac{\sqrt{\lambda_j}}{\sqrt{\lambda_l}} \ c_{lh}^j\big)\Big]X_h^c.\\
(5)\ \tilde{R}(\frac{X_i^v}{\sqrt{\lambda_i}},X_j^c)\frac{X_k^v}{\sqrt{\lambda_k}} &=\frac{1}{4} \sum_{h=1}^{n} \Big[\sum_{l=1}^{n}\big(\frac{\sqrt{\lambda_l}}{\sqrt{\lambda_j}} \ c_{jk}^l + \frac{\sqrt{\lambda_k}}{\sqrt{\lambda_l}} \ c_{lj}^k\big) \big(\frac{\sqrt{\lambda_l}}{\sqrt{\lambda_i}} \ c_{hi}^l - \frac{\sqrt{\lambda_i}}{\sqrt{\lambda_l}} \ c_{lh}^i\big)\\
&-\big(\frac{\sqrt{\lambda_k}}{\sqrt{\lambda_i}} \ c_{li}^k - \frac{\sqrt{\lambda_i}}{\sqrt{\lambda_k}} \ c_{kl}^i\big) (c_{jl}^h - c_{lh}^j +c_{hj}^l)\\
&-2\frac{\sqrt{\lambda_k}}{\sqrt{\lambda_i}} \ c_{ij}^l \ \big(\frac{\sqrt{\lambda_k}}{\sqrt{\lambda_l}} \ c_{hl}^k - \frac{\sqrt{\lambda_l}}{\sqrt{\lambda_k}} \ c_{kh}^l\big) \Big]X_h^c.\\
(6)\ \tilde{R}(\frac{X_i^v}{\sqrt{\lambda_i}},\frac{X_j^v}{\sqrt{\lambda_j}})\frac{X_k^v}{\sqrt{\lambda_k}}&=\frac{1}{4} \sum_{h=1}^{n}\sum_{l=1}^{n}\Big[\big(\frac{\sqrt{\lambda_k}}{\sqrt{\lambda_j}} \ c_{lj}^k - \frac{\sqrt{\lambda_j}}{\sqrt{\lambda_k}} \ c_{kl}^j\big)  \big(\frac{\sqrt{\lambda_h}}{\sqrt{\lambda_i}} \ c_{il}^h - \frac{\sqrt{\lambda_i}}{\sqrt{\lambda_h}} \ c_{lh}^i\big) \\
&-\big(\frac{\sqrt{\lambda_k}}{\sqrt{\lambda_i}} \ c_{li}^k - \frac{\sqrt{\lambda_i}}{\sqrt{\lambda_k}} \ c_{kl}^i\big)  \big(\frac{\sqrt{\lambda_h}}{\sqrt{\lambda_j}} \ c_{jl}^h - \frac{\sqrt{\lambda_j}}{\sqrt{\lambda_h}} \ c_{lh}^j\big)\Big] \frac{X_h^v}{\sqrt{\lambda_h}}.\\
\end{align*}
\end{theorem}
\begin{proof}
By using Theorem (\ref{thm4.1}) and definition of the Riemannian curvature, this is a direct computation. For example for $(3)$, by definition of  the curvature tensor we have:
\begin{align*}
\tilde{R}(\frac{X_i^v}{\sqrt{\lambda_i}},X_j^c)X_k^c&=\tilde{\nabla}_{\frac{X_i^v}{\sqrt{\lambda_i}}}\tilde{\nabla}_{X_j^c}X_k^c -
\tilde{\nabla}_{X_j^c}\tilde{\nabla}_{\frac{X_i^v}{\sqrt{\lambda_i}}}X_k^c - \tilde{\nabla}_{[\frac{X_i^v}{\sqrt{\lambda_i}},X_j^c]} X_k^c\\
&=\frac{1}{2} \sum_{l=1}^{n} (c_{jk}^l - c_{kl}^j +c_{lj}^k) \tilde{\nabla}_{\frac{X_i^v}{\sqrt{\lambda_i}}} X_l^c - \frac{1}{2} \sum_{l=1}^{n}\big(\frac{\sqrt{\lambda_l}}{\sqrt{\lambda_i}} \ c_{ik}^l - \frac{\sqrt{\lambda_i}}{\sqrt{\lambda_l}} \ c_{kl}^i\big) \tilde{\nabla}_{X_j^c} \frac{X_l^v}{\sqrt{\lambda_l}}\\
&-\sum_{l=1}^n \frac{\sqrt{\lambda_l}}{\sqrt{\lambda_i}} \ c_{ij}^l \ \tilde{\nabla}_{\frac{X_l^v}{\sqrt{\lambda_l}}}X_k^c\\
&=\frac{1}{4} \sum_{h=1}^{n}\sum_{l=1}^{n} (c_{jk}^l - c_{kl}^j +c_{lj}^k) \big(\frac{\sqrt{\lambda_h}}{\sqrt{\lambda_i}} \ c_{il}^h- \frac{\sqrt{\lambda_i}}{\sqrt{\lambda_h}} \ c_{lh}^i\big) \frac{X_h^v}{\sqrt{\lambda_h}}\\
&- \frac{1}{4} \sum_{h=1}^{n}\sum_{l=1}^{n} \big(\frac{\sqrt{\lambda_l}}{\sqrt{\lambda_i}} \ c_{ik}^l - \frac{\sqrt{\lambda_i}}{\sqrt{\lambda_l}} \ c_{kl}^i\big) \big(\frac{\sqrt{\lambda_h}}{\sqrt{\lambda_l}} \ c_{jl}^h + \frac{\sqrt{\lambda_l}}{\sqrt{\lambda_h}} \ c_{hj}^l\big) \frac{X_h^v}{\sqrt{\lambda_h}}\\
&-\frac{1}{2} \sum_{h=1}^{n}\sum_{l=1}^n \frac{\sqrt{\lambda_l}}{\sqrt{\lambda_i}} \ c_{ij}^l \big(\frac{\sqrt{\lambda_h}}{\sqrt{\lambda_l}} \ c_{lk}^h - \frac{\sqrt{\lambda_l}}{\sqrt{\lambda_h}} \ c_{kh}^l\big) \frac{X_h^v}{\sqrt{\lambda_h}}\\
&=\frac{1}{4} \sum_{h=1}^{n}\sum_{l=1}^{n} \Big[ (c_{jk}^l - c_{kl}^j +c_{lj}^k) \big(\frac{\sqrt{\lambda_h}}{\sqrt{\lambda_i}} \ c_{il}^h - \frac{\sqrt{\lambda_i}}{\sqrt{\lambda_h}} \ c_{lh}^i\big) \\
&-\big(\frac{\sqrt{\lambda_l}}{\sqrt{\lambda_i}} \ c_{ik}^l - \frac{\sqrt{\lambda_i}}{\sqrt{\lambda_l}} \ c_{kl}^i\big) \big(\frac{\sqrt{\lambda_h}}{\sqrt{\lambda_l}} \ c_{jl}^h + \frac{\sqrt{\lambda_l}}{\sqrt{\lambda_h}} \ c_{hj}^l\big)\\
&-2\frac{\sqrt{\lambda_l}}{\sqrt{\lambda_i}} \ c_{ij}^l \ \big(\frac{\sqrt{\lambda_h}}{\sqrt{\lambda_l}} \ c_{lk}^h - \frac{\sqrt{\lambda_l}}{\sqrt{\lambda_h}} \ c_{kh}^l\big) \Big] \frac{X_h^v}{\sqrt{\lambda_h}}.
\end{align*}
\end{proof}
Now, we can state the following theorem:

\begin{theorem}
Using the hypotheses in the previous theorem, we have the following:\\
\begin{align*}
\tilde{K}(X_i^c,X_j^c)&=\frac{1}{4}\sum_{l=1}^{n}\Big( - 4 c_{lj}^j c_{li}^i - (c_{ij}^l - c_{jl}^i +c_{li}^j)  (c_{jl}^i - c_{li}^j +c_{ij}^l) -2 c_{ij}^l (c_{lj}^i - c_{ji}^l +c_{il}^j)\Big)\\
&=K^1(X_i,X_j)\\
\tilde{K}(\frac{X_i^v}{\sqrt{\lambda_i}},\frac{X_j^v}{\sqrt{\lambda_j}})&=
\frac{1}{4}\sum_{l=1}^{n}\Big(\big(\frac{\sqrt{\lambda_j}}{\sqrt{\lambda_i}} \ c_{li}^j +
\frac{\sqrt{\lambda_i}}{\sqrt{\lambda_j}} \ c_{lj}^i\big)^2 - 4 c_{lj}^j  c_{li}^i \Big),\\
\tilde{K}(\frac{X_i^v}{\sqrt{\lambda_i}},X_j^c) &=\frac{1}{4} \sum_{l=1}^{n} \frac{\lambda_i}{\lambda_l} \ (c_{jl}^i)^2 - 3 \frac{\lambda_l}{\lambda_i} \ (c_{ij}^l)^2 - 2c_{ij}^l c_{lj}^i -4 c_{lj}^j c_{li}^i.
\end{align*}
\end{theorem}
\begin{proof}
The proof is a direct computation using the definition of the sectional curvature. For example for last one, we have
$$\tilde{K}(\frac{X_i^v}{\sqrt{\lambda_i}},X_j^c) =\tilde{g}\Big( \tilde{R}\big(\frac{X_i^v}{\sqrt{\lambda_i}},X_j^c\big)X_j^c, \frac{X_i^v}{\sqrt{\lambda_i}}\Big).$$
\end{proof}
\section{\textbf{Examples}}
The purpose of this section is to examine two examples of three-dimensional Lie groups and to determine their geometric properties as well as to classify their metrics on the tangent bundles. We first define the equivalent metrics by attention to \cite{Ha-Lee}.
\begin{definition}
Let $g$ and $g'$ be two left invariant metrics on a connected Lie group $G$. We say $g'$ is equivalent up to automorphism to $g$, written $g \sim  g'$, if there exists $\tau \in Aut(\mathfrak{g})$ such that
$$[g']=[\tau]^t[g][\tau].$$
\end{definition}
\begin{lem}
Let $G$ be a connected Lie group. Suppose that $g^1$, $g^2$, $g'$ and $g''$ are left invariant Riemannian metrics on $G$ such that $g^1\sim g'$ and $g^2\sim g''$. Then $\tilde{g}$ as the lift of $g^1$ and $g^2$ is equivalent up to automorphism to $\bar{g}$ as the lift of $g'$ and $g''$, i.e., $\tilde{g} \sim \bar{g}$
\end{lem}
\begin{proof}
By definition, there exist $\tau_1, \tau_2 \in Aut(\mathfrak{g})$ such that
$$[g']=[\tau_1]^t[g^1][\tau_1] \ \ \ \ and \ \ \ [g'']=[\tau_2]^t[g^2][\tau_2].$$
Now, we define $\mathcal{T} :\tilde{\mathfrak{g}}\to \tilde{\mathfrak{g}}$ as follows
$$\mathcal{T}=\begin{pmatrix}
\tau_2&0
\\0&\tau_1
\end{pmatrix}.$$
Then, we have $[\bar{g}]=[\mathcal{T}]^t[\tilde{g}][\mathcal{T}]$. Therefore, $\tilde{g} \sim \bar{g}$.
\end{proof}
Let $G$ be a connected Lie group. Assume that $g$ and $g'$ are two left invariant Riemannian metrics on $G$ such that $g\sim g'$.Then, there exists $\tau \in Aut(\mathfrak{g})$ such that
$$[g']=[\tau]^t[g][\tau].$$
If $\nabla$ and $\nabla'$ denote the Levi-Civita connections of $(G.g)$ and $(G,g')$, respectively, then for all $X$, $Y$ and $Z \in\mathfrak{g}$ we have:
\begin{align*}
g \big( \tau (\nabla_{X}^{'} Y), \tau(Z)\big)&=g'(\nabla_{X}^{'}Y,Z)\\
&=\frac{1}{2}\big(g' (Z,[X,Y]) + g' (Y,[Z,X])- g' (X,[Y,Z]) \big) \\
&=\frac{1}{2}\big(g(\tau(Z),[\tau(X),\tau(Y)])+g(\tau(Y),[\tau(Z),\tau(X)])-g(\tau(X),[\tau(Y),\tau(Z)])\big) \\
&=g \big(\nabla_{\tau(X)}\tau(Y),\tau(Z)\big).\\
\end{align*}
Then, $\tau (\nabla_{X}^{'} Y)=\nabla_{\tau(X)}\tau(Y)$. As a result, we obtain the following proposition.
\begin{prop}
Let $G$ be a connected Lie group. Assume that $g$ and $g'$ are two left invariant Riemannian metrics on $G$ such that $g\sim g'$.
If $R$, $R'$, $K$ and $K'$ denote the curvature tensors and the sectional curvatures of $(G.g)$ and $(G,g')$, respectively, then for all $X$, $Y$ and $Z \in \mathfrak{g}$ we have:
\begin{enumerate}
  \item $\tau \big(R'(X,Y)Z\big)=R\big(\tau(X),\tau(Y)\big)\tau(Z),$
  \item $K'(X,Y)=K\big(\tau(X),\tau(Y)\big).$
\end{enumerate}
\end{prop}
\begin{example}
Let $G$ be the three-dimensional nilpotent Lie group $\Bbb{H}$ ($\Bbb{H}$ is the Heisenberg group). In this case, we can see obviously that its Lie algebra $\mathfrak{h}$ admits a basis of $\{X=\frac{\partial}{\partial x},Y=x \frac{\partial}{\partial y}+\frac{\partial}{\partial z},Z=\frac{\partial}{\partial y}\}$ such that
$$[X,Y]=Z \ ,\ [Z,X]=[Z,Y]=0.$$
As any left invariant metric on $\Bbb{H}$ is equivalent up to automorphism to a metric whose associated matrix is of the form
$$\begin{pmatrix}
\lambda & 0 & 0
\\0&\lambda &0
\\ 0 & 0 & 1
\end{pmatrix},$$
where $\lambda>0$ (see Theorem 3.3 of \cite{Ha-Lee}). Therefore, we may take $\lambda =1,2$, i.e.,
$$
g^1 \sim \begin{pmatrix}
1& 0 & 0
\\0&1 &0
\\ 0 & 0 & 1
\end{pmatrix} \ \ \ , \ \ \ g^2 \sim \begin{pmatrix}
2 & 0 & 0
\\0& 2 &0
\\ 0 & 0 &1
\end{pmatrix}.$$
Here, the linear map $\varphi:\mathfrak{g}\to \mathfrak{g}$ is represented by a matrix concerning the basis $\{X=\frac{\partial}{\partial x},Y=x \frac{\partial}{\partial y}+\frac{\partial}{\partial z},Z=\frac{\partial}{\partial y}\}$ as follows:
$$\varphi = \begin{pmatrix}
2 & 0 & 0
\\0&2&0
\\ 0 & 0 & 1
\end{pmatrix}.$$
Then the eigenvalues of $\varphi$ are $\lambda_1=\lambda_2=2$ and $\lambda_3=1$. So for the six dimensional Lie group $T\Bbb{H}$ we have:
$$\tilde{g} \sim diag(2,2,1,1,1,1),$$
where we have used the notation $diag(a_1,\cdots, a_n)$ to denote the diagonal matrix with entries $a_1,\cdots,a_n$.
It is also possible to compute the curvature tensor and sectional curvature of $(T\Bbb{H},\tilde{g})$, for example:
$$\tilde{R}(\frac{X^v}{\sqrt{2}},Y^c)Z^c=0 , \ \ \ \ \tilde{K}(\frac{Y^v}{\sqrt{2}},Z^v)=\frac{1}{8}.$$
\end{example}
\begin{example}
Let $G$ be the three-dimensional unimodular solvable Lie group $\Bbb{R}^2\rtimes \Bbb{R}^{+}$. One can see
its Lie algebra is of the form $\mathfrak{g}=\Bbb{R}^2\rtimes \Bbb{R}$, and it has the following basis
$$\{X=\Big(\begin{bmatrix}
1
\\
0
\end{bmatrix},0\Big),Y=\Big(\begin{bmatrix}
0
\\
1
\end{bmatrix},0\Big),Z=\Big(\begin{bmatrix}
0
\\
0
\end{bmatrix},1\Big)\},$$
and we have $[X,Y]=0$, $[Z,X]=X$ and $[Z,Y]=-Y$. Any left invariant metric on $\Bbb{R}^2  \rtimes \Bbb{R}^{+}$ is equivalent, up to automorphism, to a metric whose associated matrix is of the form
$$\begin{pmatrix}
1 & 0 & 0
\\0&1 &0
\\ 0 & 0 & \nu
\end{pmatrix} \ \ \ or \ \ \ \begin{pmatrix}
1 & 0 & 0
\\0&\mu&0
\\ 0 & 0 & \nu
\end{pmatrix}$$
where $\nu >0$ and $\mu >1$ (see Theorem 3.4 of \cite{Ha-Lee}). Hence, we can take $\nu = 1,3$ and $\mu =2$, i.e.,
$$
g^1 \sim \begin{pmatrix}
1& 0 & 0
\\0&1 &0
\\ 0 & 0 & 1
\end{pmatrix} \ \ \ , \ \ \ g^2 \sim \begin{pmatrix}
1 & 0 & 0
\\0& 2 &0
\\ 0 & 0 & 3
\end{pmatrix}.$$
Hence, the linear map $\varphi:\mathfrak{g}\to \mathfrak{g}$ has a matrix representation with respect to the above basis as follows:
$$\varphi = \begin{pmatrix}
1 & 0 & 0
\\0&2&0
\\ 0 & 0 & 3
\end{pmatrix}.$$
So the eigenvalues of $\varphi$ are $\lambda_1=1$, $\lambda_2=2$ and $\lambda_3=3$. Therefore for the six dimensional Lie group $TG =T(\Bbb{R}^2  \rtimes \Bbb{R}^{+})$ we have:
$$\tilde{g} \sim diag(1,2,3,1,1,1).$$
Here, one can also compute the curvature tensor and sectional curvature of $(T\Bbb{H},\tilde{g})$. For example:
$$\tilde{R}(\frac{X^v}{\sqrt{2}},\frac{Y^v}{\sqrt{2}})Z^v=0 , \ \ \ \ \tilde{K}(Z^v,\frac{X^v}{\sqrt{2}})=\frac{1}{12}.$$
\end{example}
\begin{remark}
Let $\tilde{g}$ be a left invariant metric on $TG =T(\Bbb{R}^2  \rtimes \Bbb{R}^{+})$ such that the complete and vertical subbundles are orthogonal to each other with respect to this metric. Now, we define two left invariant metrics $g^1$ and $g^2$ on $G=\Bbb{R}^2\rtimes\Bbb{R}^{+}$ as follows:
\begin{equation*}
\forall X,Y \in {\mathfrak{g}}\ \ \ \ ; \ \ \ \
\begin{cases}
&g^1(X,Y)=\tilde{g}(X^c,Y^c),\\
&g^2(X,Y)=\tilde{g}(X^v,Y^v),\\
\end{cases}
\end{equation*}
then $\tilde{g}$ is equivalent, up to automorphism, to a metric whose associated matrix is one of the following form:
$$diag(1,1,\nu,1,1,\nu),$$
$$diag(1,1,\nu,1,\mu,\nu),$$
$$diag(1,\mu,\nu,1,1,\nu),$$
$$diag(1,\mu,\nu,1,\mu,\nu).$$
For the previous example, the same procedure can be followed, i.e., $\tilde{g}$ is equivalent up to automorphism to a metric whose associated matrix is of the form $diag(\lambda,\lambda,1,\lambda,\lambda,1)$.
\end{remark}
\section{\textbf{The symplectic structure}}
In this section, Based on recent work by Pham \cite{Pham}, we construct a symplectic form on the tangent bundle of the Lie group G by using complete and vertical lifts of left invariant vector fields.\\
Let $\omega_1$ and $\omega_2$ be two left invariant symplectic forms on the Lie group $G$. By using the map $\pi:TG\to G$ one can pull $\omega_1$ and $\omega_2$ back to $TG$ as follows:
\begin{equation*}
\begin{cases}
&\tilde{\omega}(X^c,Y^c)=\pi^*\big(\omega_1(X,Y)\big),\\
&\tilde{\omega}(X^c,Y^v)=\pi^*\big(\omega_2(X,Y)\big),\\
&\tilde{\omega}(X^v,Y^v)=0. \ \ \ \ ; \ \ \ \ \forall X,Y \in {\mathfrak{g}}
\end{cases}
\end{equation*}
\begin{theorem}
If $\omega_1$ and $\omega_2$ are two left invariant symplectic forms on a Lie group $G$ then $\tilde{\omega}$, which is defined as above, is a left invariant symplectic form on the Lie group $TG$.
\end{theorem}
\begin{proof}
We prove that $\tilde{\omega}$ is a closed and non-degenerate form. To prove closedness, it suffices to show that the following equation holds for all left invariant vector fields $\tilde{A}$, $\tilde{B}$ and $\tilde{C}$ on $TG$:
\begin{equation}\label{6.1}
\tilde{\omega}\big([\tilde{A},\tilde{B}],\tilde{C}\big)
+\tilde{\omega}\big([\tilde{B},\tilde{C}],\tilde{A}\big)
+\tilde{\omega}\big([\tilde{C},\tilde{A}],\tilde{B}\big)=0.
\end{equation}
Equivalently, using \eqref{Lie brackets}, we can see that for all left-invariant vector fields $X$, $Y$, and $Z$ on $G$, we have:
\begin{equation}\label{6.2}
\tilde{\omega}\big([X^v,Y^v],Z^v\big)+\tilde{\omega}\big([Y^v,Z^v],X^v\big)+\tilde{\omega}\big([Z^v,X^v],Y^v\big)=0,
\end{equation}
\begin{equation}\label{6.3}
\tilde{\omega}\big([X^v,Y^v],Z^c\big)+\tilde{\omega}\big([Y^v,Z^c],X^v\big)+\tilde{\omega}\big([Z^c,X^v],Y^v\big)=0,
\end{equation}
\begin{equation}\label{6.4}
\tilde{\omega}\big([X^v,Y^c],Z^v\big)+\tilde{\omega}\big([Y^c,Z^v],X^v\big)+\tilde{\omega}\big([Z^v,X^v],Y^c\big)=0
\end{equation}
\begin{equation}\label{6.5}
\tilde{\omega}\big([X^c,Y^v],Z^v\big)+\tilde{\omega}\big([Y^v,Z^v],X^c\big)+\tilde{\omega}\big([Z^v,X^c],Y^v\big)=0,
\end{equation}
\begin{equation}\label{6.6}
\tilde{\omega}\big([X^c,Y^c],Z^v\big)+\tilde{\omega}\big([Y^c,Z^v],X^c\big)+\tilde{\omega}\big([Z^v,X^c],Y^c\big)=0,
\end{equation}
\begin{equation}\label{6.7}
\tilde{\omega}\big([X^c,Y^v],Z^c\big)+\tilde{\omega}\big([Y^v,Z^c],X^c\big)+\tilde{\omega}\big([Z^c,X^c],Y^v\big)=0,
\end{equation}
\begin{equation}\label{6.8}
\tilde{\omega}\big([X^v,Y^c],Z^c\big)+\tilde{\omega}\big([Y^c,Z^c],X^v\big)+\tilde{\omega}\big([Z^c,X^v],Y^c\big)=0,
\end{equation}
\begin{equation}\label{6.9}
\tilde{\omega}\big([X^c,Y^c],Z^c\big)+\tilde{\omega}\big([Y^c,Z^c],X^c\big)+\tilde{\omega}\big([Z^c,X^c],Y^c\big)=0.
\end{equation}
For instance, \eqref{6.2} follows from the fact that the Lie bracket of any two vertical lifts is zero. For \eqref{6.3}, we have
\begin{align*}
&\tilde{\omega}\big([X^v,Y^v],Z^c\big)+\tilde{\omega}\big([Y^v,Z^c],X^v\big)+\tilde{\omega}\big([Z^c,X^v],Y^v\big)\\
&=\tilde{\omega}\big(0,Z^c\big)+\tilde{\omega}\big([Y,Z]^v,X^v\big)+\tilde{\omega}\big([Z,X]^v,Y^v\big)\\
&=0+0+0\\
&=0.
\end{align*}
For \eqref{6.6}, we have
\begin{align*}
&\tilde{\omega}\big([X^c,Y^c],Z^v\big)+\tilde{\omega}\big([Y^c,Z^v],X^c\big)+\tilde{\omega}\big([Z^v,X^c],Y^c\big)\\
&=\tilde{\omega}\big([X,Y]^c,Z^v\big)+\tilde{\omega}\big([Y,Z]^v,X^c\big)+\tilde{\omega}\big([Z,X]^v,Y^c\big)\\
&=\pi^*\big(\omega_2([X,Y],Z)\big)+\pi^*\big(\omega_2([Y,Z],X)\big)+\pi^*\big(\omega_2([Z,X],Y)\big)\\
&=\pi^*\big(\omega_2([X,Y],Z)+\omega_2([Y,Z],X)+\omega_2([Z,X],Y)\big)\\
&=\pi^*(0)\\
&=0.
\end{align*}
Now, we prove that $\tilde{\omega}$ is non-degenerate. Suppose that $\tilde{A}$ is a left invariant vector field on $TG$, using equation \eqref{1.4}, we can write $\tilde{A}=A_1^c + A_2^v$, where $A_1$ and $A_2$ are two left invariant vector fields on $G$.\\
$(i)$ If $A_1\neq 0$, then there exists a left invariant vector field $B\neq 0$ on $G$ such that
\begin{align*}
\tilde{\omega}(\tilde{A},B^v)&=\tilde{\omega}(A_1^c +A_2^v ,B^v)\\
&=\tilde{\omega}(A_2^v,B^v)+\tilde{\omega}(A_1^c ,B^v) \neq 0.
\end{align*}
$(ii)$ If $A_1=0$, then there exists a left invariant vector field $B$ on $G$ such that
\begin{align*}
\tilde{\omega}(\tilde{A},B^c)&=\tilde{\omega}(A_1^c +A_2^v ,B^c)\\
&=\tilde{\omega}(A_2^v,B^c)+\tilde{\omega}(A_1^c ,B^c) \neq 0.
\end{align*}
\end{proof}
{\large{\textbf{Acknowledgment.}}}
The authors wish to acknowledge Professor Rui Albuquerque and Professor Martin Svensson for their help and guidance.
We are grateful to the Office of Graduate Studies of the University of Isfahan for their support.

\end{document}